\documentclass[12pt,a4paper]{amsart}

\usepackage{amsmath,amsfonts,amssymb,amsthm,mathrsfs}

\usepackage{color}

\newcommand{\R}{\mathbb{R}}

\newcommand{\un}{\mathbf{1}\!\!{\rm I}} 
\newcommand{\be}{\begin{equation}} 
\newcommand{\ee}{\end{equation}}
\newcommand{\bea}{\begin{eqnarray}} 
\newcommand{\eea}{\end{eqnarray}}
\newcommand{\bean}{\begin{eqnarray*}} 
\newcommand{\eean}{\end{eqnarray*}}
\newcommand{\rf}[1]{(\ref {#1})}

\def\dx{\,{\rm d}x}
\def\dy{\,{\rm d}y}
\def\dt{\,{\rm d}t}
\def\dr{\,{\rm d}r}

\def\dl{\,{\rm d}\lambda}

\def\e{\varepsilon}
\def\eps{\varepsilon}

\def\s{\sigma}

\def\g{\gamma}

\def\r{\varrho}
\def\xn{|\!|\!|} 
\def\mn{|\!\!|}
\def\mnp{|\!\!|_{M^{d/\alpha}}}
\def\mn2{|\!\!|_{M^{d/2}}}

\def\a{\alpha}
\def\da{\frac{d}{\a}}

\newtheorem{theorem}{Theorem}
\newtheorem{proposition}[theorem]{Proposition}
\newtheorem{lemma}[theorem]{Lemma}
\newtheorem{corollary}[theorem]{Corollary}

\theoremstyle{definition}

\theoremstyle{remark}
\newtheorem{remark}[theorem]{Remark}

\numberwithin{equation}{section}
\numberwithin{theorem}{section}

\author[P. Biler]{Piotr Biler}
\address{\small Instytut Matematyczny, Uniwersytet Wroc\l awski,
 pl. Grunwaldzki 2/4, \hbox{50-384} Wroc\-\l aw, Poland}
\email{Piotr.Biler@math.uni.wroc.pl}

\author[J. Zienkiewicz]{Jacek Zienkiewicz}
\address{\small 
 Instytut Matematyczny, Uniwersytet Wroc\l awski,
 pl. Grunwaldzki 2/4, \hbox{50-384} Wroc\-\l aw, Poland}
\email{Jacek.Zienkiewicz@math.uni.wroc.pl}

\title[Solutions of  chemotaxis model]{Blowing up radial solutions\\ in the minimal Keller--Segel model of chemotaxis} 

\begin{document}

\begin{abstract} 
We consider the simplest parabolic-elliptic model of  chemotaxis in the whole space in several dimensions. 
Criteria   for the  blowup of radially symmetric solutions  in terms of suitable Morrey spaces norms are derived. 
\end{abstract}



\date{\today}

\thanks{  
The  first named author  was partially supported by the NCN grant 
2016/23/B/ST1/00434.  
He thanks  Ignacio Guerra  for interesting conversations. 
The second author was  supported by the grant  UMO-2014/15/B/ST1/00060. 
We thank the referee for pertinent remarks. 
 \\  J. Evolution Equations, https://doi.org/10.1007/s00028-018-0469-8}

\maketitle

\baselineskip=21pt

\section{Introduction and main results}
We consider in this paper solutions that blow up in a finite time for  the   Cauchy problem in space dimensions  $d\ge 2$
\begin{align}
u_t-\Delta u+\nabla\cdot(u\nabla v)&= 0,\ \ &x\in {\mathbb R}^d,\ t>0,\label{equ}\\ 
\Delta v+u &=  0,\ \  & x\in {\mathbb R}^d,\ t>0,\label{eqv}\\
u(x,0)&= u_0(x)\ge0,\ \ &x\in {\mathbb R}^d.\label{ini}
\end{align}
One motivation to study this model comes from Mathematical Biology, where equations \rf{equ}--\rf{eqv} are a simplified (the, so-called, minimal) Keller-Segel system modelling chemotaxis, see e.g. \cite{B-AMSA,BCKSV,Lem,MS,MS2}. 
The unknown variables $u=u(x,t)$ and $v=v(x,t)$ denote the density of the population of microorganisms (e.g. swimming bacteria or slime mold),  and the density of the chemical secreted by themselves that attracts them and makes them to aggregate, respectively. 

Another important interpretation of system \rf{equ}--\rf{eqv} comes from Astrophysics, where the unknown function $u=u(x,t)$ is the density of gravitationally interacting massive particles (micro- as well as macro-) in a cloud (of atoms, molecules, dust, stars, nebulae, etc.), and $v=v(x,t)$ is the Newtonian potential (``self-consistent mean field'') of the mass distribution $u$, see \cite{Cha,CSR,B-SM,B-CM,B-AM,BHN}. 
Aggregation of those particles may lead to formation of singularities (an implosion of mass phenomenon) in finite time. 

Even if in applications $u_0\in L^1(\R^d)$, and then mass 
$$
M=\int_{\R^d}u_0(x)\dx=\int_{\R^d} u(x,t)\dx\ \ \ {\rm for\ all\ \ \ }t\in[0,T_{\rm max})
$$ 
is conserved,  we will also consider locally integrable solutions with infinite mass  like the famous Chandrasekhar steady state singular solution in \cite{Cha,CSR} for $d\ge 3$
\be
u_C(x)=\frac{2(d-2)}{|x|^2}.\label{Ch}
\ee

Our results include 
\begin{itemize} 
\item
 sufficient conditions on radial initial data which lead to a finite time blowup of solutions 
\end{itemize}
 expressed in terms of quantities related to the Morrey space norm $M^{d/2}(\R^d)$ in Theorem \ref{blth}.  
For instance, condition \rf{Wc}: $\sup_{T>0}T{\rm e}^{T\Delta}u_0(0)>C(d)$ for some $C(d)\in[1,2]$ is sufficient for the blowup of solution with the initial condition $u_0$. 
Sufficient blowup conditions expressed in terms of the radial concentration \rf{2conc}: $\xn u_0\xn>{\mathcal N}$, together with an asymptotics of the number $\mathcal N$ as $d\to\infty$, are also in Proposition \ref{thr-2}.  

Similar results for the system with modified diffusion operator 
\bea
u_t+(-\Delta)^{\alpha/2}u+\nabla\cdot(u\nabla v)&=&0,\ \ x\in {\mathbb R}^d,\ t>0,\label{equ-a}\\ 
\Delta v+u &=& 0,\ \   x\in {\mathbb R}^d,\ t>0,\label{eqv-a}
\eea
supplemented with the initial condition 
\be
u(x,0)=u_0(x)\ge 0\label{ini-a}
\ee 
will be derived and discussed in Section 3. 
In a parallel way we have also

\begin{itemize}
\item 
blowup of radial solutions with large initial data (Theorem \ref{bl-a},  $\a\in(0,2)$);

\item
a reformulation of sufficient condition for blowup of radial solutions in terms of Morrey space $M^{d/\a}(\R^d)$ norm  (Proposition \ref{comp-a}); 
\end{itemize}

For the proof of the  main result, we revisit  a classical argument of H. Fujita (applied to the nonlinear heat equation in \cite{Fu}) and reminiscent of ideas in \cite{BCKSV}, which leads to a~sufficient condition for blowup of radially symmetric solutions of system \rf{equ}--\rf{eqv}, with a significant improvement compared to  \cite{BKZ2} where local moments of solutions have been employed. Then, we derive as corollaries of condition \rf{Wc} other criteria for blowup of solutions of \rf{equ}--\rf{ini}.

\noindent 
{\bf Notation.}  
The $L^p(\R^d)$ norm is denoted by $\|\, .\, \|_p$, $1\le p\le\infty$. 
The homogeneous Morrey spaces of measures on  $\R^d$ are defined by their norms 
\be
|\!\!| u|\!\!|_{M^p}\equiv  \sup_{R>0,\, x\in\R^d}R^{d(1/p-1)} 
\int_{\{|y-x|<R\}}|u(y)|\dy<\infty.\label{hMor}
\ee  
The radial concentration $\xn \, .\, \xn$ will  denote the quantity 
\be 
\xn u\xn \equiv   \sup_{R>0}R^{2-d}\int_{\{|y|<R\}}u(y)\dy,\label{2conc}
\ee 
and this quantity is equivalent to the norm in the space $M^{d/2}(\R^d)$ critical for system \rf{equ}--\rf{eqv}.  
We need in Section 3 another quantity which we call $\tfrac{d}{\a}$-radial concentration
\be 
\xn u\xn_{\da} \equiv   \sup_{R>0}R^{\a-d}\int_{\{|y|<R\}}u(y)\dy. \label{aconc}
\ee 
Evidently, $\xn\ .\ \xn\equiv\xn\ .\ \xn_{\frac{d}{2}}$.

By a direct calculation, we have 
$$
2\s_d=R^{2-d}\int_{\{|x|<R\}}u_C(x)\dx\ \ \ {\rm for\ each}\ \ \ R>0.
$$
Here, as usual, 
\be
\s_d=\frac{2\pi^{\frac{d}{2}}}{\Gamma\left(\frac{d}{2}\right)}
\label{pole}
\ee 
denotes  the area of the unit sphere $\mathbb S^{d-1}$ in $\R^d$.

The relation $f\approx g$ means that $\lim_{s\to\infty}\frac{f(s)}{g(s)}=1$ and  
$f\lessapprox g$ means: $\limsup_{s\to\infty}\frac{f(s)}{g(s)}\le 1$, $f\asymp g$ is used whenever $\lim_{s\to\infty}\frac{f(s)}{g(s)}\in (0,\infty)$.

\section{ Solutions blowing up in a finite time} \label{blowing}

It is well-known that if $d=2$, the condition leading to  a finite time blowup, i.e. 
$$\limsup_{t\nearrow T,\, x\in\R^d}  u(x,t) =\infty\ \ {\rm for\ some}\ \ 0<T<\infty,$$    
 is expressed in terms of mass,  that is   $M>8\pi$, see   e.g. \cite{B-CM,BKZ,BKZ-NHM}. 

If $d\ge 3$,  a sufficient condition for blowup  for an initial condition (not necessarily radial)   is that $u_0$ is highly concentrated, namely
\be
\left(\frac{\int_{ \R^d} |x|^\gamma u_0(x) \dx}{\int_{ \R^d}u_0(x)\dx}\right)^{\frac{d-2}{\gamma}}\le \tilde c_{d,\g} M, \label{bl-JEE}
\ee
for some  $0<\gamma\le 2$ and  a (small, explicit) constant $\tilde c_{d,\g}>0$, see \cite[(2.4)]{BK-JEE}. 
Since $$|\!\!| u_0|\!\!|_{M^{d/2}}\ge \tilde C_{d,\g}M\left(\frac{M}{\int_{\R^d}|x|^\g u_0(x)\dx}\right)^{\frac{d-2}{\g}}$$
for some constant $\tilde C_{d,\g}>0$ and all $u_0\in M^{d/2}\cap L^1$, see \cite[(2.6)]{BK-JEE}, this means that the Morrey space $M^{d/2}$ norm of $u_0$ satisfying condition \rf{bl-JEE} must be (very!) large: $$|\!\!| u_0|\!\!|_{M^{d/2}}\ge \frac{\tilde C_{d,2}}{\tilde c_{d,2}}.$$ 
According to \cite{B-CM},  $\tilde c_{d,2}=\left(2^{d/2}d\s_d\right)^{-1}$ and $\frac{\tilde C_{d,2}}{\tilde c_{d,2}}=\left(\frac{d-2}{d}\right)^{d/2-1}2^{d/2}\s_d\approx \frac{2^{d/2}}{\rm e}\s_d$.

Recently, some new results on the blowup of solutions to problem \rf{equ}--\rf{ini} appeared in \cite{K-O,BKZ,BKZ-NHM,BCKZ,BKZ2} with a new strategy of the  proofs involving local momenta of (most frequently) radial solutions, and with improved sufficient conditions in terms of the initial datum $u_0$.  
We will apply the classical proof of blowup in the seminal paper \cite{Fu} by H.~Fujita, and then improve  the sufficient conditions for the blowup expressed in terms of a~functional norm of $u_0$. 

First, we note a general property of potentials of radial functions

\begin{lemma}\label{potential}
Let $u\in L^1_{\rm loc}(\R^d)$ be a radially symmetric function,  
 such that  $v=E_d\ast u$ with $E_2(x)=-\frac1{2\pi}\log|x|$ and $E_d(x)=\frac1{(d-2)\sigma_d}|x|^{2-d}$ for $d\ge 3$, solves the Poisson equation $\Delta v+u=0$. Then the identity 
$$
\nabla v(x)\cdot x=-\frac1{\sigma_d}|x|^{2-d}\int_{\{|y|\le |x|\}} u(y)\dy
$$
holds. 
\end{lemma}

\begin{proof}
By the Gauss theorem,  
we have for the distribution function $M$ of $u$
\be
M(R)\equiv \int_{\{|y|\le R\}} u(y)\dy=
-\int_{\{|y|=R\}} \nabla v(y)\cdot\frac{y}{|y|}\,{\rm d}S.\label{distr-u}
\ee
Thus,  for the radial function $\nabla v(x)\cdot\frac{x}{|x|}$ and $|x|=R$,  we obtain the required identity
$$\nabla v(x)\cdot x=
\frac{1}{\sigma_d}R^{2-d}\int_{\{|y|=R\}}\nabla v(y)\cdot\frac{y}{|y|}\,{\rm d}S = -\frac1{\sigma_d}R^{2-d}M(R).$$ 
 \end{proof}

Now, we proceed to apply the classical idea of blowup proof in  \cite{Fu}.  

\begin{theorem}\label{blth}
 Let $d\ge 2$. 
If the inequality $T{\rm e}^{T\Delta}u_0(0) >C(d)$ holds with an explicit constant $C(d)\in[1,2]$, see \rf{Cd} below, then every radial (either classical or weak) solution of problem \rf{equ}--\rf{ini} 
 which exists on $[0,T]$ blows up in $L^\infty$ 
not later than $t=T$, i.e. $\lim_{t\nearrow T}\|u(t)\|_\infty=\infty$. 
\end{theorem}

\begin{proof} 
For a fixed $T>0$ consider the weight function $G=G(x,t)$, $x\in\R^d$, $t\in[0,T)$, which solves the backward heat equation with the  unit measure as the final time condition 
\bea 
G_t+\Delta G&=&0,\label{heat}\\ 
G(.,T)&=&\delta_0.\label{G-ini}
\eea
Clearly, we have a (unique nonnegative) solution 
\be
G(x,t)=(4\pi(T-t))^{-\frac{d}{2}}\exp\left(-\frac{|x|^2}{4(T-t)}\right), \label{G} 
\ee 
defined by the Gauss-Weierstrass kernel, satisfying $\int G(x,t)\dx =1$, 
so that  
\be
\nabla G(x,t)=-\frac{x}{2(T-t)}G(x,t).\label{gradG}
\ee 
Define for a solution $u$ of \rf{equ}--\rf{eqv} which exists on $[0,T]$ the moment 
\be
W(t)=\int G(x,t)u(x,t)\dx.\label{moment-G}
\ee
Since $G$ decays exponentially fast in $x$ as $|x|\to\infty$, the moment $W$ is well defined (at least) for  solutions $u=u(x,t)$ which are polynomially bounded in $x$. 

The evolution of the moment $W$ is governed by the identity 
\bea
\frac{{\rm d}W}{\dt}&=& \int Gu_t\dx+\int G_tu\dx\nonumber\\
&=&\int(\Delta u-\nabla\cdot(u\nabla v))G\dx-\int \Delta G\, u\dx \nonumber\\
&=&\int \Delta G\, u\dx+\int u\nabla v\cdot\nabla G\dx-\int \Delta G\, u\dx \nonumber\\
&=&-\frac{1}{2(T-t)}\int u\nabla v\cdot xG\dx\nonumber\\
&=&\frac{1}{2\s_d(T-t)}\int u(x,t)M(|x|,t)|x|^{2-d}G(x,t)\dx\label{radG}\\
&=&\frac{\s_d}{2\s_d(T-t)}\int_0^\infty\frac{1}{\s_d}M_r(r,t)r^{1-d}M(r,t)r^{2-d}G(r,t)r^{d-1}\dr\nonumber\\
&=&\frac{1}{2\s_d(T-t)}\int_0^\infty M_rMr^{2-d}G\dr\nonumber\\
&=&-\frac{1}{4\s_d(T-t)}\int_0^\infty M^2(r^{2-d}G)_r\dr,\nonumber\\
&=&\frac{1}{4\s_d(T-t)}\int_0^\infty M^2r^{1-d}\left((d-2)+\frac{r^2}{2(T-t)}\right)G\dr\nonumber
\eea
where we used the radial symmetry of the solution $u$ in \rf{radG}, Lemma \ref{potential} and, of course, the radial symmetry of $G$.
\bigskip
Expressing $W$ in the radial variables we obtain 
\bea
W(t)&=&\s_d\int_0^\infty\frac{1}{\s_d}M_rr^{1-d}Gr^{d-1}\dr\nonumber\\
&=&-\int_0^\infty MG_r\dr\nonumber\\
&=&\int_0^\infty M\frac{r}{2(T-t)}G\dr.\label{radW}
\eea
Now, applying the Cauchy inequality to the quantity \rf{radW}, we get 
\bea
W^2(t)&=&\left(\int_0^\infty M\frac{r}{2(T-t)}G\dr\right)^2\label{W2}\nonumber\\
&\le& \int_0^\infty M^2r^{1-d}\left((d-2)+\frac{r^2}{2(T-t)}\right)G\dr \label{Cauchy}\\ 
&\quad& \times \frac{1}{2(T-t)}\int_0^\infty \frac{r^{d+1}G} {r^2+2(d-2)(T-t)}\dr.\nonumber
\eea
Returning to the time derivative of $W$ in identity \rf{radG}, we arrive at the  differential inequality 
\bea
\frac{{\rm d}W}{\dt}&\ge& \frac{1}{4\s_d(T-t)}W^2(t)\left(\int_0^\infty \frac{r^{d+1}}{2(T-t)}\frac{G}{r^2+2(d-2)(T-t)}\dr\right)^{-1}\nonumber\\ 
&=&\frac{\pi^{\frac{d}{2}}}{8\s_d}W^2(t)\left(\int_0^\infty  \r^{d+1}(2(d-2)+4\r^2)^{-1}{\rm e}^{-\r^2}{\rm d}\r\right)^{-1}\nonumber
\eea
where $\r=\frac{r}{2(T-t)^{\frac12}}$. 
Recalling   \rf{pole}, we  denote 
\be
C(d)=\frac{16}{\Gamma\left(\frac{d}{2}\right)}\int_0^\infty  \r^{d+1}(2(d-2)+4\r^2)^{-1}{\rm e}^{-\r^2}{\rm d}\r.\label{Cd}
\ee
Clearly, $C(2)=2$, and $C(d)<2$ for $d\ge 3$, since we have 
$$
C(d)<\frac{16}{\Gamma\left(\frac{d}{2}\right)} \int_0^\infty\frac14 \r^{d-1}{\rm e}^{-\r^2}{\rm d}\r=\frac{4}{\Gamma\left(\frac{d}{2}\right)}\frac12 \int_0^\infty\tau^{\frac{d}{2}-1}{\rm e}^{-\tau}{\rm d}\tau=2.
$$
Thus,  we finally obtain 
\be
\frac{{\rm d}W}{\dt}\ge\frac{1}{C(d)}W^2(t)\label{evW}
\ee 
which, after an integration, leads to 
\be
W(t)\ge \left(\frac{1}{W(0)}-\frac{t}{C(d)}\right)^{-1}.\label{contr}
\ee
Now, it is clear that if 
\be
W(0)={\rm e}^{T\Delta}u_0(0)> \frac{C(d)}{T},\label{Wc}
\ee
then  $\limsup_{t\nearrow T}W(t)=\infty$ which means: $\limsup_{t\nearrow T,\, x\in\R^d}u(x,t)=\infty$, a contradiction with the existence of a locally bounded solution $u$ on $(0,T]$. 
\end{proof} 

\begin{remark}\label{rate}
The blowup rate is such that $\liminf_{t\nearrow T}(T-t)W(t)>0$. 
Indeed, 
$$\frac{1}{W(0)}-\frac{1}{W(t)}\ge \frac{t}{C(d)},$$
so if $W(0)\ge \frac{C(d)}{T}$ then 
$$W(t)\ge \frac{1}{\frac{1}{W(0)}-\frac{t}{C(d)}}\ge \frac{C(d)}{T-t}.$$ 
For other results on blowup rates (e.g. a faster blowup, the so-called, type II blowup), see \cite{GMS,MS,MS2}. 
\end{remark} 
 
\begin{remark} \label{8pi}
For $d=2$ and radially symmetric nonnegative measures $u_0$, by identity \rf{distr-u} and $C(2)=2$, condition \rf{Wc} after the integration by parts reads 
$$
\sup_{T>0}\frac{1}{4\pi}\int_0^\infty \frac{\partial}{\partial r}\left({\rm e}^{-r^2/4T}\right)M(r)\dr >2
$$
for the radial distribution function $M$ of the initial condition $u_0$. 
This means: \newline $\sup_{r>0}M(r)>8\pi$, 
and the well known blowup condition for radially symmetric solutions in $\R^2$  is recovered. 
\end{remark}

Observe that the equality  in the Cauchy inequality \rf{Cauchy} holds if and only if 
$$
0\le M(r,t)=\frac{A(t)r^d}{r^2+2(d-2)(T-t)}
=(T-t)^{\frac{d}{2}-1}\frac{A(t)2^d\r^d}{4\r^2+2(d-2)}\ \ \ {\textrm{with\ some}} \ \ \ A(t)\ge 0.
$$
Then inequality \rf{contr} reads 
\be
W(t)= \left(\frac{1}{W(0)}-\frac{t}{C(d)}\right)^{-1},\label{contr2}
\ee
and if $d\ge 3$ 
$$W(0)=\frac{1}{2T}\int_0^\infty \frac{A(0)r^{d+1}}{r^2+2(d-2)T}{\rm e}^{-r^2/(4T)}(4\pi T)^{-\frac{d}{2}}\dr 
=\frac{A(0)}{T}\frac{\Gamma\left(\frac{d}{2}\right)}{8\pi^{\frac{d}{2}}}C(d) \ge \frac{C(d)}{T},$$
then the solution blows up not later than $T$.  
This holds exactly when $A(0)\ge 4\s_d$ since \rf{pole}. 
This solution (cf. \cite[(33)]{BCKSV}) satisfies identity \rf{contr2} with $W(0)=\frac{C(d)}{T}$, and it is, in a~sense, a kind of the minimal smooth blowing up solution.  
So, we have 
 
\begin{corollary}
Moreover, if $A(t)\equiv 4\s_d$, $d\ge 3$,  we have an explicit example of blowing up solution with infinite mass 
\be
M(r,t)=\frac{4\s_d r^d}{r^2+2(d-2)(T-t)} 
\label{bl-sol}
\ee
whose density approaches $\frac{4(d-2)}{|x|^2}=2u_C(x)$, i.e. twice  the singular stationary solution, when $t\nearrow T$ so that the density of this solution  becomes infinite at the origin for $t=T$. 

Clearly, for this solution $u$ and the corresponding initial density $u_0$ we have for each $t\in[0,T)$ 
$$
u_0(x)=4(d-2)\frac{r^2+2T}{(r^2+2(d-2)T)^2} ,\ \ \ |\!\!| u_0|\!\!|_{M^{d/2}}=4\s_d=\lim_{r\to\infty}r^{2-d}M(r,t)=|\!\!|u(t)\mn2. 
$$
\end{corollary}

We express below a sufficient condition \rf{Wc} for blowup in terms of the radial concentration. 

\begin{proposition}[Comparison of blowing up solutions] \label{thr-2} 
\par\noindent 
Let $d\ge 3$ and define the threshold number 
\bea
{\mathcal N}&=&\inf\left\{N: {\rm  solution\ with\ the\ initial\ datum\ satisfying\ } M(r)=N\un_{[1,\infty)}(r) \right.\nonumber\\  &\quad& \left. {\rm blows\ up\ in\ a\ finite\ time}\right\}.\nonumber
\eea
Then  the asymptotic relation  
$${\mathcal N}\lessapprox 4\s_d\sqrt{\pi(d-2)}
$$
holds as  $d\to\infty$. 
Therefore, if $u_0\ge 0$ is such that $\xn u_0\xn > {\mathcal N}$, then the solution with $u_0$ as initial datum blows up in a finite time.
\end{proposition}

The inequality  $\xn u_0\xn > {\mathcal N}$ means that the radial distribution function corresponding to such $u_0$ satisfies $M(r)\ge R^{d-2}{\mathcal N}\un_{[R,\infty)}(r)$ for some $R>0$. 
Above,  the radial distribution function $\un_{[1,\infty)}$ corresponds, of course, to the normalized Lebesgue measure $\s_d^{-1}\,{\rm d}S$ on the unit sphere ${\mathbb S}^{d-1}$.

\begin{proof}
Here and in the sequel, due to the scaling properties of system \rf{equ}--\rf{eqv}, we may consider $R=1$ which does not lead to loss of generality.

First note that if $u_0(x)\ge 0$ is such that $N=M(1,0)>{\mathcal N}$ for the corresponding radial distribution function $M$,  then $M(r,0)\ge N\un_{[1,\infty)}(r)$ for all $r>0$ and the solution $u$ with $u_0$ as the initial datum blows up in a finite time. 
  Indeed, this is an immediate consequence of the averaged comparison principle, i.e. \cite[Theorem 2.1]{BKP} or the comparison principle for equation \rf{mass} below, see also \cite{B-BCP} in the case $d=2$
\be
\frac{\partial M}{\partial t}=M_{rr}-\frac{d-1}{r}M_r+\frac{1}{\sigma_d}r^{1-d}MM_r.\label{mass}
\ee
Thus, from equation \rf{bl-sol} we know that if 
\be 
\sup_{t>0}t{\rm e}^{t\Delta}u_0(0)>2=2\sup_{t>0}t{\rm e}^{t\Delta}\left(\frac{2(d-2)}{|x|^2}\right)(0), 
\label{Bes-est}
\ee
then $u$ blows up in a finite time. 
To check that 
\be
K_2(d)\equiv \sup_{t>0}t{\rm e}^{t\Delta}\left(\frac{2(d-2)}{|x|^2}\right)(0) =1,\label{K2}
\ee
 let us compute 
\bea 
t(4\pi t)^{-\frac{d}{2}}\int |x|^{-2}\exp\left(-\frac{|x|^2}{4t}\right)\dx &=& 
\pi^{-\frac{d}{2}}\int\frac14 |z|^{-2}{\rm e}^{-|z|^2}\,{\rm d}z \nonumber\\
&=& \frac14 \pi^{-\frac{d}{2}}\s_d\int_0^\infty {\rm e}^{-\r^2}\r^{d-3}\,{\rm d}\r \nonumber\\ 
&=& \frac{1}{4\Gamma\left(\frac{d}{2}\right)}\int_0^\infty {\rm e}^{-\tau}\tau^{\frac{d}{2}-2}\,{\rm d}\tau \nonumber\\
&=& \frac{1}{4\Gamma\left(\frac{d}{2}\right)}\Gamma\left(\frac{d}{2}-1\right)\nonumber\\
&=&\frac{1}{2(d-2)}. \label{2asy}
\eea 
Note that, by  the above computations, there exist radial initial data $u_0\in L^1(\R^d)\cap L^\infty(\R^d)$ with $W(0)$ as close to $1$ as we wish. 
In other words, we have $C(d)\in [1,2)$. 

To calculate the asymptotics of the number $\mathcal N$ observe that the quantity $\sup_{t>0}t{\rm e}^{t\Delta}u_0(0)$ in \rf{Bes-est}  for the normalized Lebesgue measure $\s_d^{-1}\,{\rm d}S$ on the unit sphere ${\mathbb S}^{d-1}$ is equal to 
\be
L_2(d)\equiv \sup_{t>0}t(4\pi t)^{-\frac{d}{2}}{\rm e}^{-\frac{1}{4t}} =\frac14 \pi^{-\frac{d}{2}} \left(\frac{d-2}{2}\right)^{\frac{d}{2}-1}{\rm e}^{1-\frac{d}{2}}.\label{2asymp} 
\ee
Therefore, by \rf{pole} and the Stirling formula for the Gamma function 
 \be
\Gamma(z+1)\approx \sqrt{2\pi z}\, z^z{\rm e}^{-z}\ \ {\rm as\ \ }z\to \infty,\label{Stir}
\ee
the asymptotic relations  $L_2(d)\approx \frac{1}{2\s_d}\frac{1}{\sqrt{\pi(d-2)}}$  and ${\mathcal N}\lessapprox 4\s_d\sqrt{\pi(d-2)}$ hold. 
\end{proof}
This improves the estimate of ${\mathcal N}\asymp d\s_d$  in \cite[Section 8]{BKZ2}.

We give below some other examples of initial data leading to a finite time blowup of solutions. 

\begin{remark}\label{cxuC}
Observe that for each initial condition $u_0\not\equiv 0$ there is $N>0$ such that condition \rf{Wc} is satisfied for $Nu_0$.

Clearly, by $\xn u_C\xn= |\!\!| u_C\mn2 =2\s_d$ and identities \rf{Bes-est}--\rf{2asy}, for  each $\eta>2$ the solution with the initial condition $u_0=\eta u_C$   blows up. 

Moreover, for each $\eta>2$ and sufficiently large $R=R(\eta)>1$ the bounded initial condition of compact support $u_0=\eta\un_{\{1\le|x|\le R\}}u_C$ leads to a blowing up solution, see \rf{Wc}. 
The singularity of that solution at the blowing up time is $\asymp\frac{1}{|x|^2}$ at the origin. 
It seems that the latter result cannot be obtained applying previously known  sufficient criteria for blowup like \rf{bl-JEE}.  

On the other hand, the initial data like $\min\{1,u_C\}+\eps \psi$ with a smooth nonnegative, compactly supported function $\psi$ and a sufficiently small $\eps>0$ (somewhere they are above the critical $u_C$ pointwisely) still  lead to global-in-time solutions according to \cite[Theorem 2.1]{BKP}. 
\end{remark}

\begin{remark}[Equivalent qualitative conditions for blowup]\label{TFCAE} \ 

 The condition $\ell(u_0)=\sup_{t>0}t{\rm e}^{t\Delta}u_0(0)>2$ is sufficient for blowup, see condition \rf{Bes-est}. 

The quantity $\tilde\ell\ge\ell$ 
\be
\tilde\ell(u_0)\equiv \sup_{t>0}t\left\| {\rm e}^{t\Delta}u_0\right\|_\infty 
\label{Besov}
\ee 
is   a Banach space norm  equivalent to the norm of the Besov space $B^{-2}_{\infty,\infty}(\R^d)$, thus for nonnegative functions $u_0$  the property $\sup_{t>0}t\left\|{\rm e}^{t\Delta}u_0\right\|_\infty\gg 1$  is equivalent to the condition $|\!\!| u_0|\!\!|_{M^{d/2}} \gg 1,$ see e.g. \cite[Prop. 2 B)]{Lem}. 
Note that, however,  the comparison constants for $\ell$, $\tilde\ell$ and $\xn \, . \, \xn$, $|\!\!|\, .\, |\!\!|_{M^{d/2}}$ strongly depend on the dimension $d$, see e.g. \cite[Proposition 7.1,Remark 8.1]{BKZ2}. 
Summarizing, qualitative sufficient conditions for blowup for radial $u_0\ge 0$ 
\begin{itemize}
\item $ \sup_{t>0}t {\rm e}^{t\Delta}u_0(0)\gg 1$,

\item $\sup_{t>0}t\left\| {\rm e}^{t\Delta}u_0\right\|_\infty\gg 1$, 

\item $\xn u_0\xn\equiv \sup_{r>0}r^{2-d}\int_{\{|x|<r\}} u_0(x)\dx \gg 1$,

\item $|\!\!|u_0\mn2 \equiv \sup_{r>0, \,x\in\R^d}r^{2-d}\int_{\{|y-x|<r\}} u_0(y)\dy \gg1$, 
\end{itemize}
are mutually equivalent, however, with comparison constants depending on  $d$. 
\end{remark}

\begin{remark}[Ill-posedness of the Cauchy problem for large data in $M^{d/2}(\R^d)$]\label{discontinuity} 
 Concerning the existence of solutions of the Cauchy problem 
\rf{equ}--\rf{ini}, we note that global-in-time mild solutions exist 
with small initial data $u_0$ in the Morrey space $M^{d/2}(\R^d)$, 
see \cite[Theorem 1 B)]{Lem}.  
Moreover, those with $u_0(x)=\frac{\e}{|x|^2}$, $0<\e\ll 1$, are selfsimilar, see \cite{B-SM}. 

Local-in-time mild solutions are shown to exist for data 
in $M^{d/2}\cap M^p(\R^d)$ with $p\in\left(\frac{d}{2},d\right)$ 
of arbitrary size, see \cite[Proposition 3.1]{BKP}.  
This assumption means that all local singularities of such data are 
strictly weaker than $\frac{1}{|x-x_0|^2}$. 
Indeed, $u_0\in M^p(\R^d)$ for $p>d/2$ implies $\lim_{r\to 0}r^{2-d}\int_{\{|x|<r\}}u_0(x)\dx=0$. 
They enjoy an instantaneous regularization property: $u(t)\in L^\infty(\R^d)$ for each $t>0$. 
More precisely,  $\sup_{0<t\le T}t^{\frac{d}{2p}}\|u(t)\|_\infty<\infty$ 
for such solutions. 
The Chandrasekhar locally unbounded solution $u_C$ is a threshold in the following sense: 
all solutions with initial data strictly below $u_C$ are global and locally bounded, see for precise statement \cite[Theorem 2.1]{BKP}. 
Therefore, the above criteria for blowup apply to solutions with 
data in $M^{d/2}\cap M^p(\R^d)\supset \{f:(1+|x|^2)f(x)\in L^\infty(\R^d)\}$ of sufficiently big size. 
Note that if a radial $u_0\in M^{d/2}(\R^d)$ has a singularity at the origin  strong enough, in the sense that 
\be
\limsup_{R\to 0}R^{2-d} \int_{\{|y|<R\}}u_0(y)\dy >\tilde{C}(d)>C(d)>2\s_d=R^{2-d} \int_{\{|y|<R\}}u_C(y)\dy\label{str-sing}
\ee
 for some large  $\tilde{C}(d)$,  then the existence time for suitable truncations of $u_0$: $\un_{\{|y|>R_n\}}u_0$, $R_n\to 0$,  tends to $0$. Therefore, a phenomenon of discontinuity of solutions with respect to the initial data occurs in $L^\infty$. 
There is no local mild solution emanating from $u_0$ that enjoy instantaneous $L^\infty$ regularization effect.   
To see this, recall from \cite{BKZ2} an estimate for the existence time of solutions of \rf{equ}--\rf{ini}. 
 In the  case of small $R>0$ in \cite[Theorem 2.9]{BKZ2} inequality \cite[(8.13)]{BKZ2} reads 
 $w_R(t)\ge CR^{d-2}\exp\left(\e R^{-2} t\right)$ for some $C>0$ independent of $R$, since under the condition 
$\limsup_{R\to 0}R^{2-d}\int_{\{|y|<R\}}u_0(y)\dy>C_d$
we have $w_R(0)\ge CR^{d-2}$. 
 Thus, $w_R(T)>M$ and blowup occurs for $T\asymp R^2$  when $R\to 0$. 
 In fact, if $u_0$ satisfies relation \rf{str-sing} then we see that any mild (hence weak) solution of system \rf{equ}--\rf{eqv} does not regularize to $L^\infty$. 
Indeed, suppose {\em a contrario} that a solution $u$ with $u_0$ as the initial data \rf{ini} is in $L^\infty$ for $t\in[t_1,t_2]$ with some $0\le t_1<t_2$. 
Assuming $t_1$ has been chosen sufficiently small, by weak continuity $u_1=u(\cdot ,t_1)$ satisfies the blowup condition \rf{str-sing} on a  ball of  fixed small radius $R>0$. 
Thus, this solution blows up before $T\asymp R^2$, so that $u$ itself blows up before $t_1+T$. 
Since we can choose sufficiently small $R>0$, there exists arbitrarily small $t_\ast>0$ such that $u$ is not in $L^\infty$ for $0<t<t_\ast$.   

Further results on the existence of global small solutions, the well-posedness (and also ill-posedness) of system \rf{equ}--\rf{ini} in Besov type spaces can be found, e.g.,  in \cite{I}. 
\end{remark}

Note that, 
there is no nonnegative  initial condition $u_0$ with   the Morrey space norm $|\!\!| u(t)|\!\!|_{M^{d/2}}$ blowing up.  
Indeed, one can prove that each nonnegative local-in-time solution of system \rf{equ-a}--\rf{eqv-a}   satisfies the condition  $\limsup_{r\to 0,\, x\in\R^d}r^{2-d}\int_{\{|y-x|<r\}}|u(y,t)|\le J(d)<\infty$  for all $t\in(0,T)$ and a universal constant $J(d)$, cf. \cite{B-book}. 
The analogue of this condition for $d=2$ has a clear meaning: the atoms of admissible nonnegative initial data $u_0$, $u_0=\lim_{t\searrow 0}u(t)$ in the sense of weak convergence of measures, are strictly smaller than $8\pi$, see \cite{BZ}.
\medskip

Our results for radially symmetric solutions in \cite{BKP} and the present paper can be summarized in the dichotomy 

\begin{corollary}\label{dich}

(i) If $u_0$ is such that $\xn u_0\xn<2\s_d$ 
then the solution of problem \rf{equ}--\rf{ini} is global-in-time; 

(ii) if $u_0$ is such that $T{\rm e}^{T\Delta}u_0(0)>2$ (so that by Proposition \ref{thr-2}  condition $2\sqrt{\pi d}2\s_d\lessapprox\xn u_0\xn$  asymptotically guarantees that), then the solution of problem \rf{equ}--\rf{ini} blows up not later than at $t=T$. 
\end{corollary}

\section{Blowup of solutions of system with fractional diffusion} 
Here, we generalize results for the Brownian diffusion case $\a=2$ to the case of  the system of nonlocal diffusion-transport equations \rf{equ-a}--\rf{ini-a}  generalizing the classical Keller-Segel system of chemotaxis to the case of  the diffusion process  given by the  fractional power of the Laplacian $(-\Delta)^{\alpha/2}$ with  $\alpha\in(0,2)$, a nonlocal operator,  as was in \cite{BCKZ,BKZ2}. 
 
System \rf{equ-a}--\rf{eqv-a} has a singular stationary solution analogous to the case of Chandrasekhar solution  for $\alpha=2$ in dimensions $d\ge 3$, cf. \cite[Th. 2.1]{BKZ}. 
\begin{proposition}[Singular stationary solutions]\label{Chandra}
Let $d\geq 2$, $2\alpha< d$, and 
\be
s(\alpha,d)=2^\alpha\frac{\Gamma\left(\frac{d-\alpha}{2}+1\right)\Gamma(\alpha)}{\Gamma\left(\frac{d}{2}-\alpha+1\right)\Gamma\left(\frac{\alpha}{2}\right)}
\approx2^{\frac{\alpha}{2}}\frac{\Gamma(\alpha)}{\Gamma\left(\frac{\alpha}{2}\right)}\s_dd^{\frac{\alpha}{2}-1},\ \ d\to\infty. \label{stala-s}
\ee 
Then $u_C(x)=\frac{s(\alpha,d)}{|x|^{\alpha}}$ is a~distributional, radial, stationary solution to system \rf{equ-a}--\rf{eqv-a}.
\end{proposition}   
This discontinuous solution $u_C\in M^{d/\a}(\R^d)$ is, in a sense, a critical one which is not smoothed out by the diffusion operator in system \rf{equ-a}--\rf{eqv-a}. 

As usual for nonlinear evolution equations of parabolic type, blowup of a solution $u$ at $t=T$ means (as in Section 2): 
$\limsup_{t\nearrow T,\, x\in\R^d}u(x,t)=\infty$. 
In fact, some $L^p$ norms (with $p>\frac{d}{\alpha}$) of $u(t)$ blow up together with the $L^\infty$-norm. 

\begin{theorem}[Blowup of solutions]\label{bl-a}
If $d\ge 2$, $\a\in(0,2)$, $T{\rm e}^{-T(-\Delta)^{\alpha/2}}u_0 (0)> C_\alpha(d)$ for a constant $C_\alpha(d)>0$ defined below in \rf{C-d-a}, then each solution of problem \rf{equ-a}--\rf{ini-a} blows up in $L^\infty$  not later than $t=T$. 
\end{theorem}

The proof of Theorem \ref{bl-a} below does not apply to the case $d=1$ which is studied by completely different methods in \cite{BC}.

Informally speaking, this sufficient condition for blowup is equivalent to $|\!\!| u_0\mnp\gg 1.$
Indeed, for radially symmetric nonnegative functions the condition  $\sup_{T>0}T{\rm e}^{-T(-\Delta)^{\alpha/2}}u_0 (0)\gg 1$
  is equivalent  to the relation $\sup_{T>0}\left\|T{\rm e}^{-T(-\Delta)^{\alpha/2}}u_0\right\|_\infty\gg 1$. 
This fact can be proved using fine estimates of the kernel of the  semigroup ${\rm e}^{-t(-\Delta)^{\a/2}}$ restricted to radial functions, as was   in the case $\a=2$.  
Further, the quantity  $\sup_{T>0}\left\|T{\rm e}^{-T(-\Delta)^{\alpha/2}}u_0\right\|_\infty$ is equivalent to the Morrey space norm $|\!\!| u_0\mnp$.  
Moreover, we note  a useful characterization of  the homogeneous Besov spaces 
$$
\sup_{T>0}T\left\|{\rm e}^{-T(-\Delta)^{\alpha/2}}u\right\|_\infty<\infty\ \ \ {\rm if\ and\ only\ if}\ \ \ u\in B^{-\a}_{\infty,\infty}(\R^d),
$$ 
shown  in \cite[Proposition 2B)]{Lem} for $\a=2$, 
and for $\a\in(0,2)$ in \cite[Sec. 4, proof of Prop. 2]{Lem2}.  
\bigskip

Before proving Theorem \ref{bl-a} we recall some analytic properties of the fractional Laplacians and the semigroups on $\R^d$ generated by them. 
The semigroup ${\rm e}^{-t(-\Delta)^{\alpha/2}}$ with $\alpha\in(0,2)$ is represented with the use of the Bochner subordination formula, cf. \cite[Ch. IX.11]{Y}
\be
{\rm e}^{-t(-\Delta)^{\alpha/2}}=\int_0^\infty f_{t,\alpha}(\lambda){\rm e}^{\lambda\Delta}\dl\label{subord}
\ee
with some functions $f_{t,\alpha}(\lambda)\ge 0$ {\em independent } of $d$. In fact, the subordinators $f_{t,\alpha}$ satisfy 
$$
{\rm e}^{-ta^\alpha}=\int_0^\infty f_{t,\alpha}(\lambda){\rm e}^{-\lambda a}\dl,
$$
so that they have selfsimilar form $f_{t,\alpha}(\lambda)=t^{-\frac{1}{\alpha}}f_{1,\alpha}\left(\lambda t^{-\frac{1}{\alpha}}\right)$. 
Therefore, the kernel $P_{t,\alpha}$  of ${\rm e}^{-t(-\Delta)^{\alpha/2}}$ is also of selfsimilar radial form, and can be expressed as 
\be
P_{t,\alpha}(x)=\int_0^\infty f_{t,\alpha}(\lambda)(4\pi \lambda)^{-\frac{d}{2}}{\rm e}^{-|x|^2/4\lambda}\dl = t^{-\frac{d}{\alpha}}R\left(\frac{|x|}{t^{\frac{1}{\alpha}}}\right) \label{Bochner}
\ee 
with a positive functon $R$ decaying algebraically, together with its derivatives $R'$, $R''$, $\dots$,  
$'=\frac{\partial}{\partial\r}$: 
\be
R(\r)\asymp\r^{-d-\alpha},\ \ R'(\r)\asymp \r^{-d-1-\alpha},\ \ R''(\r)\asymp r^{-d-2-\alpha}, \ \ \dots\ \ {\rm as}   \ \ \r\to\infty.  
\label{ker-as}
\ee
Here $r=|x|$ and $\r=\frac{r}{|T-t|^\frac{1}{\alpha}}$. 
In fact, $R$ satisfies $R(|x|)={\mathcal F}^{-1}\left(\exp\left(-|\xi|^\alpha\right)\right)(x)$.  
 This is  normalized so that  
\be 
 \s_d\int_0^\infty R(\r)\r^{d-1}\,{\rm d}\r=\frac{\s_d}{d}\int_0^\infty |R'(\r)|\r^{d}\,{\rm d}\r=1,\label{normalizacja}
 \ee
 with 
 \bea
R(0)&=&(2\pi)^{-d}\int\exp\left(-|\xi|^\alpha\right)\,{\rm d}\xi\nonumber\\
&=&(2\pi)^{-d}\alpha^{-1}\s_d\int_0^\infty {\rm e}^{-\tau}\tau^{d/\alpha-1}\,{\rm d}\tau=\frac{2\Gamma\left(\frac{d}{\alpha}\right)}{\alpha(4\pi)^{\frac{d}{2}}\Gamma\left(\frac{d}{2}\right)}. \label{zero}
\eea 
 For $\alpha=2$ we have, of course,  $R(\r)=(4\pi)^{-\frac{d}{2}}\exp\left(-\frac{\r^2}{4}\right)$.

\begin{proof}
As in the original reasoning of Fujita in \cite{Fu} applicable to the case $\a=2$ (in Section 2) and in \cite{Su} for a nonlinear fractional heat equation with power sources and $\a\in(0,2)$, here we consider the moment 
\be
W(t)=\int G(x,t)u(x,t)\dx\label{momentW}
\ee
with  the weight function $G=G(x,t)$ solving the backward linear fractional heat equation on $(0,T)$ 
\bea 
G_t-(-\Delta)^{\alpha/2} G&=&0,\label{heat}\\ 
G(.,T)&=&\delta_0.\label{G-ini}
\eea 
It is clear that $G$ has the selfsimilar radially symmetric form 
\be
G(x,t)=P_{T-t,\alpha}(x)=(T-t)^{-\frac{d}{\alpha}}R\left(\frac{|x|}{(T-t)^{\frac{1}{\alpha}}}\right)\label{ker-a}
\ee 
with  the same function $R$ as above.  

For radially symmetric functions $W$ in \rf{momentW} becomes 
$$
W(t)=-(T-t)^{-\frac{d}{\alpha}}\int_0^\infty M(r,t)R'(\r)(T-t)^{-\frac{1}{\alpha}}\dr
$$ 
with $M(r,t)=\int_{\{|x|<r\}}u(x,t)\dx=\s_d\int_0^ru(\r,t)\r^{d-1}{\rm d}\r$ so that $u(\r,t)=\frac{1}{\s_d}\frac{\partial}{\partial r}M(\r,t)\r^{1-d}$, $\r=|x|$. 
This is valid under a mild integrability condition on $u$: \ \ \ $\int u_0(x)(1+|x|)^{-d-\alpha}\dx<\infty$. 
In fact, if $u$ is a nonnegative solution of system \rf{equ-a}--\rf{eqv-a} on $\R^d\times [0,T)$, then for each $t\in[0,T)$ the condition $\int u(x,t)(1+|x|)^{-d-\alpha}\dx<\infty$ holds. 
Indeed, by the integral representation of the fractional Laplacian in \cite[(1.4)]{BKZ2} 
$$
-(-\Delta)^{\alpha/2}u(x,t)={\mathcal A}\left[\left(\int_{\{1\le|y|\}} +\lim_{\delta\searrow 0}\int_{\{\delta\le |y|\le 1\}}\right)\frac{u(x-y,t)-u(x,t)}{|y|^{d+\alpha}}\dy\right],
$$
for some constant ${\mathcal A}>0$, so that $\int_{\{1\le |y|\}}\frac{u(x-y,t)}{|y|^{d+\alpha}}\dy$ must be finite.

Further, we have by Lemma \ref{potential} that  
$\nabla v(x)\cdot x=-\frac1{\sigma_d}|x|^{2-d}\int_{\{|y|\le |x|\}} u(y)\dy$,  
and therefore by selfadjointness of the operator $(-\Delta)^{\a/2}$ 
\bea
\frac{{\rm d}W}{\dt}
&=&\int Gu_t\dx+G_tu\dx\nonumber\\
&=&\int(-(-\Delta)^{\a/2}u-\nabla\cdot(u\nabla v))G\dx+\int(-\Delta)^{\a/2}Gu\dx\nonumber\\
&=&\int u\nabla v\cdot \nabla G \dx\label{evolW}\\ 
&=&-\frac{1}{\s_d}(T-t)^{-\frac{d+1}{\alpha}} \int_0^\infty \frac{\partial}{\partial r}MMr^{1-d}R'(\r)\dr\nonumber\\
&=&\frac{1}{\s_d}(T-t)^{-\frac{d+1}{\alpha}} \int_0^\infty \frac{M^2}{2} \frac{\partial}{\partial r}(r^{1-d}R'(\r))\dr.\nonumber
\eea 
Using the Cauchy inequality as in Section 2 we estimate 
\bea 
W^2(t)&\le&(T-t)^{-\frac{d+1}{\alpha}}\int_0^\infty \frac{M^2}{2\s_d} \left|\frac{\partial}{\partial r}(r^{1-d}R'(\r))\right|\dr \nonumber\\
&\quad&\times (T-t)^{-\frac{d+1}{\alpha}}\int_0^\infty 2\s_d\frac{|R'(\r)|^2}{\left|\frac{\partial}{\partial r}(r^{1-d}R'(\r))\right|}\dr.\label{Cauchy-alpha}
\eea 
Note that the function $\r^{1-d}R'(\r)$ is strictly decreasing as the product of two strictly decreasing positive functions so that the denominator of the integrand in \rf{C-d-a} is strictly positive. 
Now, with the definition of $C_\a(d)$  
\be
C_\alpha(d)=2\s_d\int_0^\infty \frac{|R'(\r)|^2}{\left|\frac{\partial}{\partial \r}(\r^{1-d}R'(\r))\right|}{\rm d}\r,\label{C-d-a}
\ee
the ordinary differential inequality obtained from \rf{evolW} and \rf{Cauchy-alpha} 
$$ \frac{{\rm d}W}{\dt}\ge \frac{1}{C_\alpha(d)}W^2(t)
$$
leads to the estimate
$$
\frac{1}{W(0)}-\frac{1}{W(T)}\ge \frac{T}{C_\alpha(d)}.
$$
Thus, a sufficient condition for the blowup becomes 
\be
T{\rm e}^{-T(-\Delta)^{\alpha/2}}u_0 (0)>{C_\alpha(d)}.\label{bll}
\ee 
Indeed, if \rf{bll} holds, then 
$$
W(t)\ge \frac{1}{\frac{1}{W(0)}-\frac{t}{C_\a(d)}}\ge \frac{C_\a(d)}{T-t},
$$
and $\lim_{t\nearrow T}W(t)=\infty$, and therefore $\limsup_{t\nearrow T} u(x,t)=\infty$. 
\end{proof}
\medskip

Next, we express condition \rf{bll} in terms of the $\da$-concentration \rf{aconc} of  $u_0$, as was for $\a=2$ in Proposition \ref{thr-2}. 
Again, by scaling properties of system \rf{equ-a}--\rf{eqv-a}, it is sufficient to consider $R=1$.

\begin{proposition}[Comparison of blowing up solutions] \label{comp-a} 
\ 

For $d\ge 3$, $\alpha\in(0,2)$, let the threshold number ${\mathcal N}$  be 
\bea
{\mathcal N}&=&\inf\left\{N: {\rm  solution\ with\ the\ initial\ datum\ satisfying\ } M(r)=N\un_{[1,\infty)}(r) \right.\nonumber\\  &\quad&\left. {\rm blows\ up\ in\ a\ finite\ time}\right\}.\nonumber
\eea
Then  the asymptotic relation  
$${\mathcal N}\lesssim \s_dd^{\frac{\alpha}{2}} \ \ {\rm holds \ as }\ \   d\to\infty.
$$ 
\end{proposition}

\begin{proof}
Let us compute for the kernel  $P_{t,\alpha}$ of the semigroup  ${\rm e}^{-t(-\Delta)^{\alpha/2}}$ the quantity
\bea
K_\alpha(d)&=& \sup_{t>0}tP_{t,\a}\left(\frac{s(\alpha,d)}{|x|^\alpha}\right)(0) \nonumber\\ &=&s(\alpha,d) \sup_{t>0} t^{1-\frac{d}{\alpha}}\s_d\int_0^\infty R\left(\frac{r}{t^{\frac{1}{\alpha}}}\right)r^{-\alpha+d-1}\dr\nonumber\\
&=&s(\alpha,d)\s_d \int_0^\infty R(\r)\r^{d-1-\alpha}\,{\rm d}\r \label{s-alpha}\\
&=&s(\alpha,d) \s_d\int_0^\infty\int_0^\infty f_{1,\alpha}(\lambda)(4\pi)^{-\frac{d}{2}} \lambda^{-\frac{d}{2}}{\rm e}^{-\r^2/4\lambda}\r^{d-\alpha-1}\dl \, {\rm d}\r\nonumber\\
&=&s(\alpha,d) \s_d\pi^{-\frac{d}{2}}\int_0^\infty f_{1,\alpha}(\lambda)\int_0^\infty 2^{-d}{\rm e}^{-\tau}\lambda^{-\frac{d}{2}+\frac{d}{2}-\alpha/2}2^{d-\alpha-1}\tau^{\frac{d-\alpha}{2}-1}\,{\rm d}\tau \nonumber\\
&=&2^\alpha\frac{\Gamma\left(\frac{d-\alpha}{2}+1\right)\Gamma(\alpha)}{\Gamma\left(\frac{d}{2}-\alpha+1\right)\Gamma\left(\frac{\alpha}{2}\right)} \frac{\Gamma\left(\frac{d-\alpha}{2}\right)}{\Gamma\left(\frac{d}{2}\right)} 2^{-\alpha}\int_0^\infty f_{1,\alpha}(\lambda) \lambda^{-\frac{\alpha}{2}}\dl\nonumber\\
&=& k_0(\alpha)\frac{\Gamma\left(\frac{d-\alpha}{2}+1\right)}{\Gamma\left(\frac{d}{2}-\alpha+1\right)}\frac{\Gamma\left(\frac{d-\alpha}{2}\right)}{\Gamma\left(\frac{d}{2}\right)},\nonumber\\
&\approx& k(\alpha)\nonumber
\eea
for some constants $k_0(\alpha),\, k(\alpha)>0$ {\em independent} of $d$, $d\to\infty$, by formulas \rf{stala-s}, \rf{Bochner}. 

\noindent 
By the comparison principle in \cite[Th. 2.4]{BKZ2}, if $0\le u_0\le \eps u_C$ for an $\eps\in[0,1)$, then the solution is global so it does not blow up in finite time, therefore $K_\alpha(d)\le C_\alpha(d)$ by this comparison result. 

So, now we need an upper estimate of the constant $C_\alpha(d)$ defined in \rf{C-d-a}. 
By definition \rf{C-d-a}, the global-in-time existence result \cite[Th. 2.4]{BKZ2}  and relations \rf{normalizacja},  we obtain 
\be
K_\alpha(d)\le C_\alpha(d)\le \frac{2d}{d-2}.\label{as-C-d-a}
\ee 
Indeed, the left hand side inequality is the consequence of the comparison principle in  \cite[Th. 2.4]{BKZ2}. 
Then, by representation \rf{Bochner} we have $R'(\r)<0$ for $\r>0$, and 
\be
0\le \r R''(\r)-R'(\r)=\int_0^\infty f_{1,\alpha}(\lambda)(4\pi\lambda)^{-\frac{d}{2}}\left(\frac{\r^2}{4\lambda^2}-\frac{\r}{2\lambda}+\frac{\r}{2\lambda}\right){\rm e}^{-\r^2/4\lambda}\dl, \label{rR}
\ee
so that 
$$
d-1+\r\frac{R''(\r)}{|R'(\r)|}\ge d-2,$$
and the right hand side inequality in estimate \rf{as-C-d-a} follows.  
\bigskip

Now, we will  test  the normalized Lebesgue measure $\s_d^{-1}\,{\rm d}S$ on the unit sphere ${\mathbb S}^d$ corresponding to  the radial distribution function $\un_{[1,\infty)}(r)$ 
\bea
L_\alpha(d)&\equiv&\sup_{t>0}t{\rm e}^{-t(-\Delta)^{\alpha/2}}\left(\s_d^{-1}\,{\rm d}S\right)\nonumber\\ 
 &=& \sup_{t>0}t ^{1-\frac{d}{\alpha}}R\left(\frac{1}{t^{\frac{1}{\alpha}}}\right) \nonumber\\ 
&=&\sup_{\r>0}\r^{d-\alpha}R(\r)\nonumber\\
&=&\sup_{\r>0}\int_0^\infty f_{1,\alpha}(\lambda)(4\pi\lambda)^{-\frac{d}{2}}\r^{d-\alpha}{\rm e}^{-\r^2/4\lambda}\dl\nonumber\\
&=&2^{-\alpha}\pi^{-\frac{d}{2}}\sup_{\r>0}\int_0^\infty f_{1,\alpha}\left(\frac{\r^2}{4\tau}\right)\left(\frac{\r^2}{4\tau}\right)^{1-\frac{\alpha}{2}}\tau^{\frac{d-\alpha}{2}-1}{\rm e}^{-\tau}\,{\rm d}\tau. \label{alphaasy} 
\eea
From \rf{alphaasy}, the evident upper bound for $L_\alpha(d)$ is 
\bea
L_\alpha(d)&\le& 
2^{-\alpha} \pi^{-\frac{d}{2}}\sup_{x>0} f_{1,\alpha}(x) x^{1-\frac\alpha{2}} \times\int_0^\infty\tau^{\frac{d-\alpha}{2}-1}{\rm e}^{-\tau}\, {\rm d}\tau  \nonumber\\ 
&=& \tilde k(\alpha)\pi^{-\frac{d}{2}} \Gamma\left(\frac{d-\alpha}{2}\right) 
\nonumber\\
&=&\frac{2\tilde k(\alpha)}{\s_d}\frac{\Gamma\left(\frac{d-\alpha}{2}\right)}{\Gamma\left(\frac{d}{2}\right)}\nonumber\\
&\approx&2\tilde k(\alpha)\frac{1}{\s_d}d^{-\frac\alpha{2}}\label{alphaasy2}
\eea 
for some constant $\tilde k(\alpha)>0$ independent of $d$, $d\to\infty$, 
similarly as was in computations of \rf{s-alpha}, with the use of the Stirling formula  \rf{Stir}.

Now, we need an asymptotic lower bound for the quantity $L_\a(d)$. 
Observe that 
\bea
m\equiv\max_{\tau>0}{\rm e}^{-\tau}\tau^{\frac{d-\alpha}{2}-1} 
&=& {\rm e}^{-\tau_0}\tau_0^{\frac{d-\alpha}{2}-1}\ \ \ \ {\rm with}\ \ \ \  \tau_0=\left(\tfrac{d-\a}{2}-1\right) \nonumber\\ 
&=&{\rm e}^{-\frac{d-\alpha}{2}+1}\left(\frac{d-\alpha}{2}-1\right)^{\frac{d-\alpha}{2}-1}\nonumber\\
&\approx&\Gamma\left(\frac{d-\alpha}{2}\right)\frac{1}{\sqrt{\pi(d-\alpha-2)}}\label{min}
\eea
holds by \rf{Stir}. 
Now, let $h\asymp d^{\frac12}$.  
It is easy to check that 
 $$\frac{1}{m}\min_{[\tau_0,\tau_0+h]}{\rm e}^{-\tau}\tau^{\frac{d-\alpha}{2}-1} \ge \delta$$ 
 for some $\delta>0$, uniformly in $d$. 
Indeed, 
$$
\log\frac{(d+h)^d{\rm e }^{-d-h}}{d^d{\rm e}^{-d} }=d\log\left(1+\frac{h}{d}\right)-h\approx d\frac{h}{d}-\frac{dh^2}{2d^2}-h= {\mathcal O}\left(\frac{h^2}{2d}\right).
$$
From formulas \rf{alphaasy} and \rf{min} we infer 
\bea 
L_\alpha(d)&\ge& \pi^{-\frac{d}{2}}\sup_{\r>0} \int_{\tau_0}^{\tau_0+h}  f_{1,\alpha}\left(\frac{\r^2}{4\tau}\right)\left(\frac{\r^2}{4\tau}\right)^{1-\frac\alpha{2}}\tau^{\frac{d-\alpha}{2}-1}{\rm e}^{-\tau}\,{\rm d}\tau \nonumber\\
&\ge& \pi^{-\frac{d}{2}}\frac{\delta h}{\sqrt{d}}\Gamma\left(\frac{d-\alpha}{2}\right)\nonumber \\
&\approx&\pi^{-\frac{d}{2}}\delta \Gamma\left(\frac{d}{2}\right)d^{-\frac\alpha{2}}. \label{alphasy3}
\eea 
Therefore $L_\alpha(d)\ge \delta\frac{1}{\s_d}d^{-\frac\alpha{2}}$ holds. 
This is an estimate of optimal order and {\em different} from its counterpart  for $\alpha=2$.  
Remark  that if 
\be
\tilde\ell_\alpha(u_0)\equiv \sup_{t>0}t\left\|{\rm e}^{-t(-\Delta)^{\alpha/2}}u_0 \right\|_\infty,\label{l-a}
\ee
then  the comparison constants of $\tilde\ell_\a(\,.\,)$ with the $\da$-concentration $\xn\,.\,\xn_{\da}$ depend on $d$.

It is clear that if $NL_\alpha(d)\ge C_\alpha(d)$ then $N\ge {\mathcal N}$.
Thus, if the radial distribution function $M$ corresponding to the density $u_0$ satisfies 
$$
|\!\!| u_0\mnp\ge \xn u_0\xn_{\da} \ge r^{\alpha-d}M(r)>\frac{C_\alpha(d)}{L_\alpha(d)}\ \ {\rm for\ some\ \ }r>0,$$ 
then  the solution with $u_0$ as the  initial condition blows up in a finite time again by the comparison principle  \cite[Th. 2.4]{BKZ2}.  
Therefore,  by \rf{as-C-d-a}, we obtain that ${\mathcal N}=\frac{C_\alpha(d)}{L_\alpha(d)} \lesssim \s_d d^{\frac\alpha{2}}$ holds. 
\end{proof}

\begin{remark}[Examples of blowing up solutions]\label{comp-sing} 
Observe that for any initial condition $u_0\not\equiv 0$ there is $N>0$ such that \rf{bll} is satisfied for $Nu_0$. 

Similarly as was in Remark \ref{cxuC}   for $\a=2$, if  $u_0(x)=\eta u_C(x)$, then the sufficient condition for blowup \rf{bll} is satisfied for large $d$ whenever $\eta> \frac{1}{k(\alpha)}\frac{2d}{d-2}$. 
Indeed, it suffices to have $\eta>\frac{C_\alpha(d)}{K_\alpha(d)}$, and by relation  \rf{as-C-d-a} asymptotically  $\frac{C_\alpha(d)}{K_\alpha(d)}\lesssim \frac{1}{k(\alpha)}\frac{2d}{d-2}\approx \frac{2}{k(\a)}$ as $d\to\infty$. 

More generally than in  Remark  \ref{cxuC}, 
for each such $\eta$ and sufficiently large $R=R(\eta)>1$, the bounded initial condition of compact support $u_0=\eta\un_{\{1\le|x|\le R\}}u_C$ leads to a blowing up solution. 
It seems that this result cannot be obtained applying previous sufficient criteria for blowup involving moments in \cite{BK-JEE} and  in \cite[Th. 2.9]{BKZ2}.
\end{remark}

Taking into account \cite[Theorem 2.1]{BKZ},  all the above remarks on the critical  Morrey space and the $\da$-radial concentration, we formulate the following   dichotomy result 
\begin{corollary}\label{dichot}
Let $d\ge 2$, $\alpha\in(1,2]$ and $d+1>2\alpha$. 
There exist  
two positive constants $c(\alpha,d)$ and $C(\alpha,d)$ such that 
\begin{itemize}

\item[(i)] if $\alpha\in(1,2)$ and $\xn u_0\xn_{\frac{d}{\a}}<c(\a,d)$ then problem \rf{equ-a}--\rf{ini-a} has a~global-in-time solution;

\item[(ii)]
  $\xn u_0\xn_{\da} >C(\alpha,d)$ implies that each nonnegative radially symmetric  solution of problem \rf{equ-a}--\rf{ini-a} blows up in a finite time.
\end{itemize}  
\end{corollary}

This, together with \rf{alphaasy2},  shows that for $\alpha\in(1,2)$  the discrepancy between bounds of the $\da$-radial concentration  sufficient for either  global-in-time existence or  the finite time  blowup, i.e. $\frac{C(\alpha,d)}{c(\alpha,d)}$,  is of order $d^{\frac\alpha2}$, similarly as was established in \cite[Rem. 8.1]{BKZ2} using an analysis of moments of solutions defined with compactly supported weight functions.  
Indeed, $C(\alpha,d)=\frac{C_\alpha(d)}{L_\alpha(d)}$ and $c(\alpha,d)\ge \xn u_C\xn_{\da}=|\!\!| u_C|\!\!|_{M^{d/\alpha}}\asymp \s_dd^{\frac{\alpha}{2}-1}$.

In the case $\alpha=2$,  Proposition \ref{thr-2}   gives a better result: 
the discrepancy between the bounds of the radial concentration sufficient for either the  global-in-time existence or for the finite time  blowup  is of order $d^{\frac12}$, which improves the result in     \cite[Remark 8.1]{BKZ2} where this quotient has been shown to be of order $d$.

\end{document}